\documentclass[10pt]{amsart}

\usepackage{amsfonts}
\usepackage{amsthm}
\usepackage{amssymb}
\usepackage{amsmath}

\title[General Bufetov]{Decoding Rauzy Induction: An Answer to Bufetov's General Question}
\author[Fickenscher]{Jon Fickenscher}
\date{\today}

\newtheorem{thm}{Theorem}[section]
\newtheorem{lem}[thm]{Lemma}
\newtheorem{cor}[thm]{Corollary}
\newtheorem{main}{Main Theorem}
\newtheorem*{mainlem}{Main Lemma}
\newtheorem*{nnthm}{Theorem}
\newtheorem{ques}{Question}
\theoremstyle{definition}
\newtheorem{defn}[thm]{Definition}
\newtheorem{rem}[thm]{Remark}

\newcommand{\DDD}{\mathcal{D}}
\newcommand{\III}{\mathcal{I}}
\newcommand{\NNN}{\mathcal{N}}
\newcommand{\OOO}{\mathcal{O}}
\newcommand{\PPP}{\mathcal{P}}

\newcommand{\CC}{\mathbb{C}}
\newcommand{\NN}{\mathbb{N}}

\newcommand{\RR}{\mathbb{R}}
\newcommand{\ZZ}{\mathbb{Z}}

\newcommand{\PF}{Perron-Frobenius}

\newcommand{\eps}{\varepsilon}

\newcommand{\perm}{\mathfrak{S}}
\newcommand{\irr}{\mathfrak{S}^0}

\newcommand{\SPAN}{\mathrm{span}}

\newcommand{\RHScase}[1]{\left\{\begin{array}{ll} #1 \end{array}\right.}

\begin{document}

\begin{abstract}
         Given a typical interval exchange transformation,
	    we may naturally associate to it an infinite sequence of matrices
	    through Rauzy induction.
	  These matrices encode visitations of the induced interval exchange transformations
	    within the original.
	  In 2010, W.\ A.\ Veech showed that these matrices suffice to 
	      recover the original interval exchange transformation,
	      unique up to topological conjugacy,
	      answering a question of A.\ Bufetov.
	  In this work, we show that interval exchange transformation
	      may be recovered and is unique modulo conjugacy when we instead only know consecutive products of
	      these matrices.
	  This answers another question of A.\ Bufetov.
	  We also extend this result to any inductive scheme that
	      produces square visitation matrices.
\end{abstract}

\maketitle

\section{Introduction}

      Interval exchange transformations (IET's) are invertible piece-wise translations on an interval $I$.
      They are typically defined by a permutation $\pi$ on $\{1,\dots,n\}$ and a choice of partitioning
	  of $I$ into sub-intervals $I_1,\dots,I_n$ with respective lengths $\lambda_1,\dots,\lambda_n$.
      The sub-intervals are reordered by $T$ according to $\pi$.

      Rauzy induction, as defined in \cite{cRau1979}, is a map that sends an IET $T$ on $I$ to
	  its first return $T'$ on $I' \subset I$ for suitably chosen $I'$.
      For almost every\footnote{For every appropriate $\pi$ and Lebesgue almost every $\lambda = (\lambda_1,\dots,\lambda_n)\in \RR_+^n$.} IET $T$,
	  Rauzy induction may be applied infinitely often.
      This yields a sequence $T^{(k)}$, $k\geq 0$, of IET's so that each transition $T^{(k-1)}\mapsto T^{(k)}$
	  is the result of a Rauzy induction.
      To each step we may define a \emph{visitation matrix} $A_k$ so that $(A_k)_{ij}$ counts the number of disjoint images of the intervals $I_j^{(k)}$
		in $I_i^{(k-1)}$ before return to $I^{(k)}$.
      It is part of the general theory of IET's that the initial $\pi$ and the sequence of $A_k$'s define $T$ uniquely up to
	      topological conjugacy.
      In preparation for \cite{cBuf2013}, A. Bufetov posed the following.
      \begin{ques}[A. Bufetov]
	  Given only the sequence of $A_k$'s, can the initial permutation $\pi$ be determined and is it unique?
      \end{ques}
      In response, W. A. Veech gave an affirmative answer in \cite[Theorem 1.2]{cVe2010}.
      This allowed A. Bufetov to ensure the injectivity of a map that intertwines the Kontsevich-Zorich cocycle
	      with a renormalization cocycle (see the remark ending Section 4.3.1 in \cite{cBuf2013}).

      However, if another induction scheme was used to get visitation matrices,
	  we may not know each individual $A_k$.
      For instance, we may follow A. Zorich's acceleration of Rauzy induction (see \cite{cZor1996})
	  or choose to induce on the first interval $I_1$.
      In either of these cases,
	    our visitation matrix $B$ will actually be a product
	    $A_1\cdots A_N$ of the $A_k$'s realized by Rauzy induction.
      Motivated by this, we say that a sequence $B_\ell$, $\ell\in \NN$, is a product of the $A_k$'s if
	      there exist an increasing sequence of integers $k_\ell$, $\ell\geq 0$, so that
	      $k_0 = 0$ and
	    $B_\ell = A_{k_{\ell-1}+1}A_{k_{\ell-1}+2}\cdots A_{k_{\ell}}$
	    for each $\ell\geq 1$.
      We now are able to pose Bufetov's second, more general, question.
      \begin{ques}[A. Bufetov]
	  Given instead a sequence $B_\ell$, $\ell\in \NN$, of products of the $A_k$'s, can the initial permutation $\pi$ still be determined and is it unique?
      \end{ques}

      This work is dedicated to answering this second question and its generalizations.
      We answer in the affirmative by our main results.
      Extended Rauzy induction is more general than regular Rauzy induction
	  and is discussed in Section \ref{SSecLRI} before Lemma \ref{LemIDOCext}.

      \begin{main}
	  If $B_1,B_2,B_3,\dots$ are consecutive matrix products defined by an infinite sequence of steps of
		(extended) Rauzy induction,
		then the initial permutation $\pi$ is unique.
      \end{main}

      Recently, J. Jenkins proved this result in \cite{cJenkinsThesis} for the $3\times 3$ matrix case.
      He then explored the $4\times 4$ case numerically.

      In the most general case,
	  we call the inductions on $I'\supsetneq I'' \supsetneq I'' \supsetneq \dots$ an
	  \emph{admissible induction sequence}
	  if the $n\times n$ visitation matrices $A_k$ from $T^{(k-1)}$ to $T^{(k)}$ are well-defined.
      We then are able to answer Bufetov's question in a much broader setting.
      
      \begin{main}
	If visitation matrices $B_1,B_2,B_3,\dots$ are defined by an admissible induction sequence,
	      then the initial permutation $\pi$ is unique.
      \end{main}

     \subsection*{Outline of Paper}

      In Section \ref{SecDef} we establish our notation and provide known results concerning
	  IET's and related objects as well as general linear algebra.
	  In particular, the anti-symmetric matrix $L_\pi$ is defined given $\pi$,
	    and this matrix plays a central role here.
      In Section \ref{SecPF} the \PF\ eigenvalue and eigenvector are discussed.
      The main argument of that section is Corollary \ref{CorPosNeverNull},
		which says that the \PF\ eigenvector cannot be in the
		nullspace of any linear combination $L_\pi-c L_{\pi'}$ for
		permutations $\pi,\pi'$ and scalar $c$.
      Section \ref{SecMain} begins with a reduction of Main Theorem 1 to
		a special case, stated as the Main Lemma.
		The section ends with a proof of the Main Lemma.
      Section \ref{SecMain2} reduces Main Theorem 2 to Main Theorem 1 by Lemma \ref{LemAdmissIsRauzy}.
		This lemma states that any admissible induction sequence
		    must arise from extended Rauzy induction.
      Appendix \ref{SecA} provides further results concerning admissibility
		and induced maps which lead to the proof of Lemma \ref{LemAdmissIsRauzy}
		in Appendix \ref{SecB}.


\section{Definitions}\label{SecDef}

	 An \emph{interval} or \emph{sub-interval} is of the form $[a,b)$
		  for $a<b$,
	  i.e. a non-empty subset of $\RR$ that is closed on the left and open on the right.
	  If $I = [a,b)$ is an interval, $|I| = b-a$ denotes its \emph{length}.
	    For a set $C$, $\#C$ denotes its cardinality.
	  A \emph{translation} $\phi:I \to J$ for intervals $I$ and $J$
	      is any function that may be expressed as $\phi(x) = x + c$ for constant $c$.
	  If $\psi:C\to D$ is a function and $E\subseteq C$ we use the notation $\phi E$ to mean
	    the \emph{image of $E$ by $\psi$}, or
		  $ \psi E = \{\phi(c): c\in E\}\subseteq D.$
	  For $\lambda\in \RR_+^n$, or a vector in $\RR^n$ with all positive entries,
		  $ |\lambda| = \lambda_1 + \dots \lambda_n$
	  denotes the $1$-norm of $\lambda$.

	  \subsection{Permutations and a Matrix}

	  The notations in this section describe either standard definitions
		from algebra or standard literature on interval exchange transformations.
	  Let $\perm_n$ be the set of all permutations on $\{1,\dots,n\}$, i.e. bijections on $\{1,\dots,n\}$.
	  \begin{defn}\label{DefIrr}
		The \emph{irreducible permutations} on $\{1,\dots,n\}$, $\irr_n$, is the set of $\pi\in \perm_n$
		      so that
			$\pi\{1,\dots,k\} = \{1,\dots,k\}$ iff $k = n$.
	  \end{defn}

	  \begin{defn}\label{DefLPi}
	      For $\pi\in \perm_n$, the anti-symmetric $n\times n$ matrix $L_\pi$ is given by
			$$ (L_{\pi})_{ij} = \RHScase{1, & i < j \mbox{ and } \pi(i) > \pi (j),\\
						    -1, & i > j \mbox{ and } \pi(i) < \pi(j),\\
						     0, & \text{otherwise,}}$$
		$1\leq i,j,\leq n$.
	  \end{defn}

	  The proof of the Main Theorem requires that no distinct $\pi,\pi' \in \irr_n$
		satisfy $L_\pi = L_{\pi'}$.
	  This is given by the next result.

	  \begin{lem}\label{LemLPi}
	      The map from $\perm_n$ to the set of $n\times n$ matrices given by
		    $$ \pi \mapsto L_\pi$$
	      is injective.
	  \end{lem}

	  \begin{proof}
		The result follows immediately from the relationship
		      $$ \pi(i) - i = \sum_{\pi(j) \leq \pi(i)} 1 - \sum_{k\leq i}1
			  = \sum_{j=1}^n \chi_{\pi(i)\geq \pi(j)}(j) - \chi_{i\geq j}(j)= \sum_{j=1}^n (L_\pi)_{ij},$$
		for all $i\in\{1,\dots,n\}$, where $\chi$ is the indicator function.
	  \end{proof}

	  The final two definitions simply fix notation of established concepts from linear algebra
		and will be used without remark for what follows.

	  \begin{defn}
	      If $L$ is an $n\times n$ matrix, $\NNN_L$ will denote its \emph{nullspace}, i.e.
		      $$ \NNN_L = \{v\in \CC^n: Lv =0\}.$$
	      For $\pi\in \irr_n$, $\NNN_\pi = \NNN_{L_\pi}$.
	  \end{defn}

	  \begin{defn}
	      If $L$ is an $n\times n$ anti-symmetric matrix, $(\cdot,\cdot)_L$ is the \emph{bilinear form
			    associated to $L$} given by
			$$ (u,v)_L = u^* L v,$$
	      where the last value is treated as a scalar.
	      For $\pi \in \irr_n$, $(\cdot,\cdot)_\pi = (\cdot,\cdot)_{L_\pi}$.
	  \end{defn}

	\subsection{Interval Exchange Transformations}

	      An interval exchange transformation $T$ is an invertible transformation on an interval
		  that divides the interval into sub-intervals of lengths $\lambda_1,\dots,\lambda_n$ and
		  reorders them according to $\pi$.
	      We will assume $n \geq 2$, as $T$ is the identity if $n=1$.
	      
	      More precisely, for fixed $\pi\in \irr_n$ and $\lambda \in \RR_+^n$,
			let $\beta_j = \sum_{i \leq j} \lambda_i$ for $0\leq j \leq n$
			and $I = [0,\beta_n)$,
			where $\beta_0 =0$ and $\beta_n = |\lambda|$.
	     For each interval $I_j= [\beta_{j-1},\beta_j)$, then $T$ restricted to $I_j$ is just translation by a
		      value $\omega_j$.
	      If the $j^{th}$ interval \emph{is in position}\footnote{Many texts on interval exchange transformations let $\pi(j)$ describe
		      the interval in position $j$ after the application of $T$.} $\pi(j)$ after the application of $T$,
		      then
		      $ \omega_j =  \sum_{\pi(i) < \pi(j)} \lambda_i- \sum_{k< j} \lambda_k.$
	      We see that $\omega = L_\pi \lambda$.
	      \begin{defn}
		  The \emph{interval exchange transformation (IET)} defined by $(\pi,\lambda)$, $T=\III_{\pi,\lambda}$,
			is a map $I \to I$ defined piece-wise by
			    $$ T(x) = x + \omega_j,\text{ for }x\in I_j,$$
		    $1\leq j \leq n$.
	      \end{defn}
	      We restrict our attention to $\pi\in \irr_n$ when defining an IET.
	      Indeed, if $\pi \in \perm_n \setminus \irr_n$ then there exists $k<n$ so that
		    $T[0,\beta_k) = [0,\beta_k)$.
	      In this case we may reduce to studying $T$ restricted to $[0,\beta_k)$
		    and $[\beta_k,\beta_n)$ separately.

		\begin{defn}\label{DefIDOC}
			IET $T = \III_{\pi,\lambda}$ satisfies the \emph{infinite distinct orbit condition or  i.d.o.c.}\footnote{This is also
			  known as the \emph{Keane Condition}.} iff
			  each orbit $\OOO_T(\beta_j) = \{T^k \beta_j: k\in \ZZ\}$, $1\leq j <n$,
			  is infinite and the orbits are pairwise distinct.	 
		\end{defn}

		  If $\lambda$ is rationally independent, meaning $c_1\lambda_1 + \dots + c_n \lambda_n = 0$
		      has a solution with $c_i$'s integers iff $c_1 = \dots = c_n =0$,
		      then $T = \III_{\pi,\lambda}$ is an i.d.o.c. IET.
		  Therefore, $T$ is is i.d.o.c. for fixed $\pi$ and Lebesgue almost every $\lambda\in \RR_+^n$.

		  \begin{lem}[Keane \cite{cKea1975}]\label{LemMinimal}
			If $T$ is an i.d.o.c. IET on $I$, then for any sub-interval $J\subseteq I$, $\bigcup_{k=0}^\infty J = I$,
			and for any $x\in I$ the orbit $\OOO_T(x)$ is infinite and dense.
		  \end{lem}

	  \subsection{Admissible Inductions}

	    Consider any interval $I'\subseteq I$ for IET $T = \III_{\pi,\lambda}$,
		$\pi\in \irr_n$ and $\lambda\in \RR_+^n$.
	    Informally, we will call $I'$ \emph{admissible} if the induced map
	      $T'$ is an IET on $n$ intervals and an $n\times n$ visitation matrix
	      $A$ is well defined.
	    Let $r(x) = \min\{k\in \NN: T^kx\in I'\}$ be the \emph{return time} of $x\in I'$.
	    Then the \emph{induced transformation} for $T$ on $I'$ is denoted by
	      $T|_{I'}$ and is given by
	      $$ T|_{I'}(x) = T^{r(x)}(x)$$
	     for each $x\in I'$.
	    \begin{defn}\label{DefAdmiss}
		Let $T= \III_{\pi,\lambda}$ be an i.d.o.c. $n$-IET on $I$.
		The sub-interval $I'\subseteq I$ is \emph{admissible} for $T$ if
		there exists a partition of $I'$ into $n$ consecutive sub-intervals
		$I_1',I_2',\dots,I_n'$ so that for each $1\leq i \leq n$:
		\begin{enumerate}
		      \item $r(x) = r(x')$ for each $x,y\in I'_i$,
			  letting $r_i$ be this common value,
		      \item for each $0\leq k < r_i$, $T$ restricted to $T^kI'_i$ is
			  a translation,
		      \item for each $0 \leq k < r_i$, $T^k I'_i \subseteq I_j$ for some
			  $1\leq j \leq n$.
		  \end{enumerate}
		  The $n\times n$ \emph{visitation matrix} $A$ is given by $A_{ij} = \#\{0\leq k < r_i: T^k I_j' \subseteq I_i\}$.
	      \end{defn}

	    This definition of admissible is equivalent to the one given in \cite[Section 3.3]{cDolPer2013} for
		 i.d.o.c. $T$.
	      It follows that if $I'$ is admissible for $T$, then $T' = T|_I'$
		    is an $n$-IET,
		  and
		      $$ \lambda = A \lambda'$$
	      where $T' = \III_{\pi',\lambda'}$ and
		       $\lambda'_j = |I'_j|$ for $1\leq j \leq n$.

	      \begin{rem}\label{RemInduce}
	      Consider i.d.o.c. $T$ on $I$ with $I''\subseteq I' \subseteq I$.
	      It is a consequence that any of the two following statements imply the third:
	      \begin{enumerate}
		  \item $I'$ is admissible for $T$ on $I$,
		  \item $I''$ is admissible for $T$ on $I$,
		  \item $I''$ is admissible for $T'$ on $I'$, where $T'$ is the induction of $T$ on $I'$.
	      \end{enumerate}
	      Also, if $A_1$ is the visitation matrix of the induction from $I$ to $I'$ and
	      		$A_2$ is the visitation matrix of the induction from $I'$ to $I''$,
	      		then the product $A_1A_2$ is the visitation matrix 
	      		of the induction from $I$ to $I''$.
	      \end{rem}

	      Please refer to Appendix \ref{SecA} for a more thorough discussion of
		      admissible inductions.
	      For example, Lemma \ref{LemA1} proves the remark assuming the first statement and one other holds.

	\subsection{Rauzy Induction}

	      Rauzy induction was defined in \cite{cRau1979},
		and we see that it is defined as an admissible induction over an
		appropriately chosen sub-interval $I'$.
	      We recall the definition here and discuss some results relevant for our work.

	      For $\pi \in \irr_n$, let $m = \pi^{-1}(n)$ denote the interval placed last by $\pi$.
	      Assume that $\lambda_n \neq \lambda_m$ and let $I' = [0, \beta_n - \min\{\lambda_m,\lambda_n\})$.
	      The induced transformation $T'= T|_{I'}$ is also an IET and we give the
		      description below for  $T'= \III_{\pi',\lambda'}$.
	      If $\lambda_n > \lambda_m$, then
			  $$ \pi'(i) = \RHScase{\pi(i), & \pi(i) < \pi(n),\\
					  \pi(n) +1, & \pi(i)=n,\\
					  \pi(i)+1, & \pi(i) > \pi(n),
				}
				\text{ and }
			    \lambda'_i = \RHScase{ \lambda_n - \lambda_m, & i = n,\\
				  \lambda_i, & i < n,
				}$$
	      for $1\leq i \leq n$.
	      If $\lambda_n<\lambda_m$, then instead
			  $$ \pi'(i) = \RHScase{\pi(i), & i \leq m,\\
					  \pi(n), & i = m+1,\\
					  \pi(i-1), & i > m+1,
				}
				\text{ and }
			    \lambda'_i = \RHScase{ \lambda_i, & i <m,\\
					\lambda_m - \lambda_n, & i = m,\\
					\lambda_n, & i = m+1,\\
					\lambda_{i-1}, & i >m+1,
				}$$
	      for $1\leq i \leq n$.

		\begin{defn}
		      Consider $T$ defined by $\pi$ and $\lambda$ with $m$ as above.
			If $\lambda_n>\lambda_m$, the change from $T$ to $T'$
			  is a move of \emph{Rauzy induction of type $0$}.
			If $\lambda_n<\lambda_m$, the change from $T$ to $T'$
			  is a move of \emph{Rauzy induction of type $1$}.
			If $\lambda_n = \lambda_m$, Rauzy Induction is not
			  well defined.
		\end{defn}

		For fixed $\pi\in \irr_n$, the condition $\lambda_n = \lambda_m$ is of zero Lebesgue measure
		      in $\RR_+^n$.
		Given $\pi$ and the type of Rauzy induction $\eps$, we may define the visitation matrix $A = A_{(\pi,\eps)}$
		      as given in Definition \ref{DefAdmiss}.
		If $\eps = 0$, then
			$$ A_{ij} = \RHScase{1, & i = j,\\
					1, & i=n,~j = m,\\
					0, & \text{otherwise,}}$$
		 and if $\eps = 1$, then
			$$ A_{ij} = \RHScase{1, & i=j<m,\\
				      1, & j=i+1 >k,\\
				      1, & i =n,~j = m,\\
				      0, & \text{otherwise.}}$$
		  Suppose we may act by $N$ consecutive steps of Rauzy induction on $T = \III_{\pi,\lambda}$,
			and let $T,T',T'',\dots, T^{(N)}$ be the resulting IET's at each step where
			$T^{(k)} = \III_{\pi^{(k)},\lambda^{(k)}}$.
		  Let $\eps_k$ be the type of induction from $T^{(k-1)}$ to $T^{(k)}$.
		  If
			    $$ B = A_{(\pi,\eps_1)}A_{(\pi',\eps_2)}\cdots A_{(\pi^{(N-1)},\eps_N)}$$
		  then
			$ \lambda = B \lambda^{(N)}.$

		  We may verify that if $A= A_{\pi,\eps}$ and $\pi'$ is the result of the type $\eps$
			induction on $\pi$, then
			$ A^* L_\pi A = L_{\pi'}$.
		  For a proof of this with different notation, see \cite[Lemma 10.2]{cVia2006}.
		  It follows that if $B$ is defined by $N$ consecutive steps of induction with initial
			permutation $\pi$ and ending at $\pi^{(N)}$, then
			    \begin{equation}\label{EqBLB}
			       B^* L_{\pi} B = L_{\pi^{(N)}}.
			    \end{equation}
		  We finish this section by answering a question: for which IET's can Rauzy induction be applied infinitely many times?
		  See Section 4 of \cite{cYoc2006} for a treatment of this result.
		  \begin{lem}\label{LemIDOC}
		      If $\pi\in \irr_n$, $\lambda\in \RR_+^n$ and $T = \III_{\pi,\lambda}$, then the following are equivalent:
			      \begin{enumerate}
				  \item\label{LemIDOCa} $T$ is i.d.o.c., and
				  \item\label{LemIDOCb} $T$ admits infinitely many steps of Rauzy induction.
			      \end{enumerate}
		  \end{lem}

		  The following is shown in Sections 1.2.3--1.2.4 of \cite{cMarMouYoc2005}.
		  \begin{lem}\label{LemIDOC2}
			If $T = \III_{\pi,\lambda}$ admits infinitely many steps of Rauzy induction
				    and $A_k = A_{(\pi^{(k)},\lambda^{(k)})}$ are the corresponding matrices,
			      then for each $j\in \NN$ there exists $k_0 = k_0(j)\in \NN$ so that
			      for all $k>k_0$,
				    $$ A_{[j,j+k]} = A_k A_{k+1} \cdots A_{j+k}$$
			      is a matrix with all positive entries.
		  \end{lem}

	\subsection{Left Rauzy Induction}\label{SSecLRI}
	Left Rauzy induction was defined in \cite{cVe1990} and describes
	    inducing on $T = \III_{\pi,\lambda}$ by $I' = [\min\{\lambda_1,\lambda_{m'}\},|\lambda|)$,
	      where $m' = \pi^{-1}(1)$ for $\pi\in \irr_n$ and $\lambda\in \RR_+^n$.
	In other words, instead of removing a sub-interval from the right as in (right) Rauzy induction,
	      we remove one from the left.
	We will give the explicit definitions and then show how this type of induction relates to
	      (right) Rauzy induction.

	  \begin{defn}
	      \emph{Left Rauzy induction} is the result of taking the first return of $T = \III_{\pi,\lambda}$ on $I'$
		      as defined above.
		The induction is \emph{type $\tilde{0}$} iff $\lambda_1>\lambda_{m'}$ and \emph{type $\tilde{1}$}
		    iff $\lambda_1< \lambda_{m'}$.
		The induction is not well defined if $\lambda_1 = \lambda_{m'}$.
	  \end{defn}
	Let $T' = \III_{\pi',\lambda'}$ be the resulting IET by this induction
	      (up to a translation so that $I'$ begins at $0$).
	If the induction is type $\tilde{0}$, then
	      $$
		\pi'(i) = \RHScase{	\pi(1)-1, & i = m',\\	
				  \pi(i)-1, & 1<\pi(i)<\pi(i),\\
				  \pi(i), & \pi(i) \geq \pi(1),
				  }
		    \text{ and }
			  \lambda'_i = \RHScase{ \lambda_1 - \lambda_{m'}, & i= 1,\\
				      \lambda_i, & i>1.
			  }
	      $$
	  Likewise, if the induction is type $\tilde{1}$, then
	      $$
		  \pi'(i) = \RHScase{\pi(i+1), & i < m'-1,\\
				  \pi(1), &  i = m'-1,\\
				  \pi(i), & i \geq m',\\
				  }
		    \text{ and }
		    \lambda_i' = \RHScase{
				    \lambda_{i+1}, & i< m'-1,\\
				    \lambda_{1}, & i= m'-1,\\
				    \lambda_{m'} - \lambda_1, & i= m',\\
				    \lambda_i, & i > m'.}
		$$

	    Let $\tau_n$ be given by
		$$
		    \tau_n(i) = n - i\text{ for } 0\leq i \leq n.
		$$
	    For $\pi\in \irr_n$, let $\pi_\tau$ be given by
		$$ \pi_\tau = \tau_{n+1} \circ \pi \circ \tau_{n+1},$$
	    noting that $\pi_\tau\in \irr_n$ as well.
	    If $\tilde{\eps} \pi$ is the result of type $\tilde{\eps}$ induction on $\pi$
		and $\eps\pi_\tau$ is the result of type $\eps$
		induction on $\pi_\tau$,
		then
		$$
		\tilde{\eps}\pi = (\eps \pi_\tau)_\tau.
		$$
		For $\lambda\in\RR_+^n$, let $\lambda_\tau$ be given by
		$$
		    (\lambda_\tau)_i = \lambda_{\tau_{n+1}(i)}.
		$$
		If $\tilde{\eps}\lambda$ and $\eps\lambda_\tau$ are defined analogously to $\tilde{\eps} \pi$
		    and $\eps\pi_\tau$,
		    then
		    $$ \tilde{\eps}\lambda = (\eps\lambda_\tau)_\tau.$$
		We define the $n\times n$ permutation matrix $P_n$ by
		$$
		    (P_n)_{ij} = \RHScase{1, & i = \tau_{n+1}(j),\\
			    0, & \text{otherwise,}}
		$$
		then we see that $A_{\pi,\tilde{\eps}} = P_n A_{\pi_\tau,\eps} P_n$,
		where $A_{\pi,\tilde{\eps}}$ is the visitation matrix that satisfies
		$\lambda = A_{\pi,\tilde{\eps}} \cdot \tilde{\eps}\lambda$.
		Furthermore,
		  $$ L_{\pi} = P_n L_{\pi_\tau} P_n,$$
		and so $A^*_{\pi,\tilde{\eps}} L_\pi A_{\pi,\tilde{\eps}} = L_{\tilde{\eps}\pi}$
		as a direct consequence.

		Therefore, if $A_1,\dots, A_N$ are visitation matrices given by
		    consecutive steps of \emph{extended Rauzy induction}, i.e.
		    left and/or right Rauzy induction,
		    and $B = A_1\cdots A_N$ is the product,
		    then
		      $$
			  B^* L_\pi B = L_{\pi^{(N)}}
		      $$
		     where $\pi$ is the initial permutation and $\pi^{(N)}$ is the
		     resulting permutation after the $N$ steps.

		  The proof of the following is a modification of Lemma \ref{LemIDOC2}
		    and has a similar proof.
		  However, the notation from \cite{cMarMouYoc2005} is significantly different and will not be included here.
		  \begin{lem}\label{LemIDOCext}
			If $T = \III_{\pi,\lambda}$ admits infinitely many steps of extended Rauzy induction
				    and $A_k$ are the corresponding matrices,
			      then for each $j\in \NN$ there exists $k_0 = k_0(j)\in \NN$ so that
			      for all $k>k_0$,
				    $$ A_{[j,j+k]} = A_k A_{k+1} \cdots A_{j+k}$$
			      is a matrix with all positive entries.
		  \end{lem}

     \subsection{Veech's result for $\NNN_\pi$}\label{SSecVeech}
	  The main result in this section is shown in \cite[Lemma 5.7]{cVe1984i}.
	  Please refer to that work as well as \cite{cVe1982}
		for the original definitions and proofs.

	  Consider each $\pi\in \irr_n$ to be extended
		so that $\pi(0) = 0$ and $\pi(n+1) = n+1$.
	  Note that $\pi_\tau$ respects this extension as well.
	  Let $\sigma_\pi$ be a function on $\{0,\dots,n\}$ given by
	      $$
		  \sigma_\pi(i) = \pi^{-1}(\pi(i) + 1) -1,
	      $$
	  as in \cite{cVe1982}.
	  Let $\Sigma(\pi)$ be the partition of $\{0,1,\dots,n\}$ given by orbits of $\sigma_\pi$.
	  For each $S\in \Sigma(\pi)$, let $b_S\in \ZZ^n$ be given by
	      $$
		    (b_S)_i = \chi_S(i-1) - \chi_S(i).
	      $$
	  It was shown in \cite[Lemma 5.3]{cVe1984i} that
	      $$
		    \#\Sigma(\pi) = \dim\NNN_\pi + 1,
	      $$
	    and \cite[Proposition 5.2]{cVe1984i} states that
	      $$
		    \SPAN\{b_S:S\in \Sigma(\pi)\} = \NNN_\pi.
	      $$

	  \begin{lem}[Veech \cite{cVe1984i}]\label{LemVeechPaper}
		  For $\pi\in \irr_n$ and $\eps\in \{0,1\}$, there exists
		      a bijection $\eps:\Sigma(\pi) \to \Sigma(\eps\pi)$ so that
			      $$ A_{(\pi,\eps)} b_S = b_{\eps S}$$
		  for each $S\in \Sigma(\pi)$.
	  \end{lem}

	  Recall $\tau_n$, $\tau_{n+1}$ and $P=P_n$ from the previous section.
	  By direct computation, we see that
	    $$
		\sigma_{\pi_\tau} = \tau_n \circ \sigma_{\pi}^{-1} \circ \tau_n,
	    $$
	    and so $\Sigma(\pi_\tau) = \tau_n \Sigma(\pi)$.

	  \begin{cor}\label{CorVeechPaper}
		  For $\pi\in \irr_n$ and $\eps\in \{0,1\}$, there exists
		      a bijection $\tilde{\eps}:\Sigma(\pi) \to \Sigma(\eps\pi)$ so that
			      $$ A_{(\pi,\tilde{\eps})} b_S = b_{\tilde{\eps} S}$$
		  for each $S\in \Sigma(\pi)$.
	  \end{cor}

	  \begin{proof}
		For each $S\in \Sigma(\pi)$ and $1\leq i \leq n$,
		    $$ \begin{array}{rcl}
			      (Pb_S)_i  & = & (b_S)_{n+1-i}\\
					& = & \chi_S(n-i) - \chi_S(n+1-i)\\
					& = & \chi_{\tau_n S}(\tau_n(n-i)) - \chi_{\tau_n S}(\tau_n(n+1-i))\\
					& = & \chi_{\tau_n S}(i) - \chi_{\tau_n S}(i-1)\\
					& = & - (b_{\tau_n S})_i.
		       \end{array}$$
		And so
			$$ A_{(\pi,\tilde{\eps})} b_S = P A_{(\pi_\tau,\eps)} P b_S = - P A_{(\pi_\tau,\eps)} b_{\tau_nS}.$$
		By Lemma \ref{LemVeechPaper} we continue,
			$$ - P A_{(\pi_\tau,\eps)} b_{\tau_nS} = -P b_{\eps (\tau_n S)} = b_{\tau_n(\eps(\tau_nS))}.$$
		Therefore the desired bijection is $\tilde{\eps} = \tau_n \circ \eps \circ \tau_n$.
	  \end{proof}

	\subsection{Invariant Spaces}

	For $n\times n$ matrix $B$, a subspace $V\subseteq \CC^n$ is \emph{$B$-invariant}
	      if $BV \subseteq V$.
	If $B$ is invertible then $V$ is $B$-invariant iff $BV = V$.
	An \emph{eigenbasis} of $B$ for $V$ is a basis $\{u_1,\dots,u_m\}$ of $V$
		  so that each $u_j$ is an eigenvector for $B$.
	Recall that an \emph{eigenvector} for $B$ is a non-zero vector $u$ with
		a corresponding \emph{eigenvalue} $\alpha$ such that $u\in \NNN_{B_\alpha^p}$
	for some $p\in\NN$ where $B_\alpha = B - \alpha I$ for identity matrix $I$.
	The lemma and corollary in this section allow us to find an eigenbasis for $\CC^n$ that
	    includes bases of invariant subspaces.
	The definition that follows then correctly associates to a $B$-invariant subspace
	      eigenvalues.

	\begin{lem}
	    Let $B$ be an $n\times n$ matrix and $V,W\subseteq \CC^n$ be subspaces such that $W\subseteq V$
		and $V,W$ are each $B$-invariant.
	    There exists an \emph{eigenbasis} $\{u_1,\dots,u_m\}$ of $B$ of $V$ so that
		$\{u_1,\dots,u_{m'}\}$ is a basis for $W$, $m' = \dim W$.
	\end{lem}

	\begin{cor}\label{CorSplit}
	    If $V,W\subseteq \CC^n$ are $B$-invariant subspaces, $m = \dim V$ and $m' = \dim W$, then there exists
		  an eigenbasis $\{u_1,\dots,u_n\}$ of $B$ for $\CC^n$ such that
		  \begin{itemize}
		   \item $\{u_{n-m-m'+\ell+1},\dots, u_{n-m'+\ell}\}$ is a basis for $W$,
		   \item $\{u_{n-m'+1},\dots, u_{n-m'+\ell}\}$ is a basis for $V\cap W$,
		   \item $\{u_{n-m'+1},\dots, u_n\}$ is a basis for $W$,
		  \end{itemize}
		  where $\ell = \dim(V\cap W)$.
	\end{cor}

	\begin{defn}
	    If $V\subseteq \CC^n$ is $B$-invariant, and $\alpha_1,\dots,\alpha_m$ are the respective eigenvalues
		for the eigenbasis in the previous lemma or corollary, then they are the \emph{eigenvalues of $B$ over $V$}.
	\end{defn}

\section{The \PF\ Eigenvalue}\label{SecPF}

	  We begin with a specific case of a fundamental result.
	  See \cite[Theorem 0.16]{cWalters2000} for a more general version of this theorem.

	  \begin{nnthm}[\PF\ Theorem]
		If $B$ is a positive matrix, then there exists positive eigenvalue $\alpha$ for $B$
		      so that for all other eigenvalues $\alpha'$ of $B$, $\alpha>|\alpha'|$.
		Furthermore, there exists a positive eigenvector $u$ for $B$ with eigenvalue $\alpha$
		      and any eigenvector $u'$ for $B$ with eigenvalue $\alpha$ is a scalar multiple of $u$.
	  \end{nnthm}

	  We call $\alpha$ the \emph{\PF\ eigenvalue}
		such a positive vector $u$ a \emph{\PF\ eigenvector}.
	  If $B$ is a positive integer matrix, then $\alpha>1$.
	  Corollary \ref{CorNullspace} tells us that $u$ is not in $\NNN_\pi$ for any $\pi\in \irr_n$.
	  Then Corollary \ref{CorPosNeverNull} forbids $u$ from being in $\NNN_{L}$ for any
		    non-zero matrix $L = L_\pi - c L_{\pi'}$.
	  Finally, Corollary \ref{CorPFUnique} tells us that for a fixed eigenbasis $\{u_1,\dots,u_n\}$ of $B$
		  with $u_1 = u$ there exists a unique $u_j$ so that $(u_1,u_j)_\pi\neq 0$.

	  \begin{defn}
		An \emph{extended Rauzy cycle} at $\pi$ is a finite sequence of consecutive steps of
		    extended Rauzy induction that begins and ends at $\pi$.
	  \end{defn}

	  \begin{lem}\label{LemVeech}
		If $B$ is described by an extended Rauzy cycle at $\pi\in \irr_n$, then
		      there exists a basis $\{b_1,\dots,b_m\}$ of $\NNN_\pi$ and $p\in \NN$ so that
		      $$ B^p b_i = b_i,$$
		for each $i\in \{1,\dots,m\}$.
	  \end{lem}

	  \begin{proof}
		 As in Section \ref{SSecVeech}, let $b_S$ for $S\in \Sigma(\pi)$.
		By applying Lemma \ref{LemVeechPaper} and Corollary \ref{CorVeechPaper} to the product $B$,
		      we have a bijection $d$ on $\Sigma(\pi)$ so that
		      $ B b_S = b_{dS}$ for each $S\in \Sigma(\pi)$.

		Let $p$ be any power such that $d^p$ is the identity on $\Sigma(\pi)$, and
		    choose the $b_1,\dots,b_m$ as a subset of the $b_S$'s that form a basis of $\NNN_\pi$.
		Then for any $i\in \{1,\dots,m\}$,
			  $$ B^p b_i = B^p b_S = b_{d^pS} = b_S = b_i, $$
		  where $b_i = b_S$.
	  \end{proof}

	  \begin{cor}\label{CorNullspace}
		Let $B$ be given by an extended Rauzy cycle at $\pi$.
		If $\beta_1,\dots,\beta_m$ are the eigenvalues of $B$ over $\NNN_\pi$, $\pi\in \irr_n$,
		      then each $\beta_j$ is a root of unity.
		Furthermore, if $B$ is a positive integer matrix with \PF\ eigenvalue $\alpha>1$,
			then $\{\beta_1,\dots,\beta_m\}\cap\{\alpha,1/\alpha\} = \emptyset$.
	  \end{cor}

	    \begin{lem}\label{LemPositiveNeverNull}
		    If $L = L_{\pi} - c L_{\pi'}$ for distinct $\pi,\pi'\in \irr_n$ and real $c$,
			  then there exists $i$ such that row $i$ of $L$ is non-zero and
			  either non-positive or non-negative.
	    \end{lem}

	    \begin{proof}
		    We consider the value of $c$.
		    If $c\leq 0$, then the first row of $L$ is at least the first row of $L_{\pi}$.
		    The claim then holds for $i=1$.
		    If $c>1$ and $L' = L_{\pi'} - \frac{1}{c} L_{\pi}$
		    then $L = -c L'$.
		    If we find a non-zero row for $L'$, then the same row satisfies the
		      claim for $L$ (but with the opposite sign).
		    We therefore consider two remaining cases: $0<c<1$ and $c=1$.

		    If $0<c<1$, let $i$ satisfy $\pi(i) = n$.
		    Because $\pi\in \irr_n$, $i<n$.
		    Then if $i>j$,
			  $$ L_{ij} = (L_\pi)_{ij} - c (L_{\pi'})_{ij} = - c (L_{\pi'})_{ij} \geq 0.$$
		    If $i<j$,
			  $$ L_{ij} = (L_\pi)_{ij} - c (L_{\pi'})_{ij} = 1 - c (L_{\pi'})_{ij} >0.$$
		    Therefore row $i$ is non-zero and is non-negative.

		    If $c=1$, then let $i$ be such that $\pi(i) \neq \pi'(i)$ and that maximizes $\pi(i)$.
		    Let $k\leq n$ be the position of $i$ in $\pi$, i.e. $k = \pi(i)$.
		    Note that $\pi'(i)<k$.
		    If $\pi(j) > k$, then $\pi(j) = \pi'(j)$ and so
			      $$ (L_\pi)_{ij} = (L_{\pi'})_{ij} \Rightarrow L_{ij} = 0.$$
		    If $\pi(j) < k$ and $i>j$, then
			      $$ L_{ij} = (L_\pi)_{ij} - (L_{\pi'})_{ij} = -(L_{\pi'})_{ij} \geq 0.$$
		    If $\pi(j) < k$ and $i<j$, then
			      $$ L_{ij} = (L_\pi)_{ij} - (L_{\pi'})_{ij} = 1 - (L_{\pi'})_{ij} \geq 0.$$
		    Row $i$ of $L$ is non-negative, and we must verify that it is non-zero.
		    Let $i'\neq i$ satisfy $\pi'(i') = k$.
		    By definition, $\pi(i')<\pi(i)$ and $\pi'(i') > \pi'(i)$.
		    Either
			  \begin{itemize}
			   \item $i<i'$ and so $(L_\pi)_{ii'} = 1$ and $(L_{\pi'})_{ii'} =0$, or
			    \item $i>i'$ and so $(L_\pi)_{ii'} = 0$ and $(L_{\pi'})_{ii'} = -1$.
			  \end{itemize}
		    Therefore, $L_{ii'} =1$ and row $i$ of $L$ is non-zero.
	    \end{proof}

	    \begin{cor}\label{CorPosNeverNull}
		  If $v$ is a positive vector, then it does not belong to the nullspace of
			  $$ L = L_\pi - c L_{\pi'}$$
		  for distinct $\pi,\pi'\in \irr_n$ and complex $c$.
	    \end{cor}

	    \begin{proof}
		  Suppose $c$ is real.
		  By Lemma \ref{LemPositiveNeverNull}, there exists a row $i$ that is non-zero and
			non-negative (resp. non-positive).
		  Therefore $(L v)_i >0$ (resp. $(Lv)_i<0$) and so $Lv \neq 0$.
		  If $c$ is non-real, then by Lemma \ref{LemPositiveNeverNull} the real vector
			      $L_{\pi'} v$ is non-zero.
		  Therefore
		      $$ L v = (L_\pi - c L_{\pi'}) v$$
		  must have a non-zero imaginary component and cannot be zero.
	    \end{proof}

	  \begin{lem}\label{LemBF}
		Let $B^*LB = L$ for anti-symmetric matrix $L$ and matrix $B$.
		If $u,u'$ are eigenvectors of $B$ with corresponding eigenvalues $\alpha,\alpha'$,
			and $(u,u')_L \neq 0$, then $\overline{\alpha} \alpha' = 1$.
	  \end{lem}

	  \begin{proof}
		Recall that the \emph{order} of eigenvalue $u$ with eigenvalue $\alpha$ is the minimum
		      $p\in \NN$ so that $u\in \NNN_{B_\alpha^p}$, $B_\alpha = B - \alpha I$.
		We proceed by induction, first on the order of $u$ and then on the sum of the orders
		    of $u$ and $u'$.

		If $u,u'$ are both \emph{true eigenvectors}, i.e. order $1$, then
		      $$ (u,u')_L = (Bu,Bu')_L = \overline{\alpha}\alpha' (u,u')_L,$$
		and $\overline{\alpha} \alpha' =1$ as $(u,u')_L \neq 0$ by assumption.

		If $u$ is higher order and $u'$ is order $1$,
			let $w$ be defined by $A u = \alpha u +w$
			and note that $w$ is one order lower than $u$ for the same
			eigenvalue $\alpha$.
		Then
		  $$ (u,u')_L = (Bu,Bu')_L = \overline{\alpha}\alpha' (u,u')_L + \alpha' (w,u')_L.$$
		If $(w,u')_L\neq 0$, then the claim follows by induction.
		If $(w,u')_L = 0$, then the claim follows just as in the base case.

		If $u,u'$ are both of higher order, let $w$ be defined as before and let $w'$ be defined by
		      $Bu' = \alpha' u' + w'$.
		  Then $(u,u')_L = (Bu,Bu')_L$ which is the sum
		    $$ \overline{\alpha}\alpha' (u,u')_L + \alpha' (w,u')_L
			    + \overline{\alpha} (u,w')_L +  (w,w')_L.$$
		  If any term but the first is non-zero, then the claim follows by induction.
		    Note that $(v,v')_L = \overline{(v',v)_L}$.
		  If only the first term is non-zero, then the claim is verified as before.
	  \end{proof}

	  \begin{cor}\label{CorPFUnique}
	      Let $B^* L B = L$ for anti-symmetric matrix $L$ and positive matrix $B$.
	      Let $u_1,\dots,u_n$ an eigenbasis for $B$ with respective eigenvalues $\alpha_1,\dots,\alpha_n$
		  such that $u_{n-m+1},\dots,u_n$ forms a basis of $\NNN_L$ of dimension $m < n$.
	      If $\alpha_1>0$ is the \PF\ eigenvalue with positive eigenvector $u_1$
		  (in particular $u_1\notin \NNN_L$),
		  then there is a unique $j \leq n-m$ so that $(u_1,u_j)_L \neq 0$.
	      This is also the unique $j \leq n-m$ so that $\alpha_j = 1/\alpha_1$.
	  \end{cor}

	  \begin{proof}
	      The final claim follows from the first by Lemma \ref{LemBF}.
	      Because $u_1 \notin \NNN_L$,
		    there must exist $u_j$ so that $(u_1,u_j)_L = c \neq 0$.
	      By Lemma \ref{LemBF}, $\alpha_j = 1/\alpha_1$.
	      Because $u_j\notin \NNN_L$, $j\leq n-m$.

	      Suppose by contradiction there exists $j' \neq j$ so that $(u_1,u_{j'})_L = c'\neq 0$.
	      Again $j'\leq n-m$ and by Lemma \ref{LemBF}, $\alpha_{j'} = \alpha_j$.
	      If there exists $i \neq 1$ so that either $(u_i,u_{j'})_L \neq 0$
		  or $(u_i,u_j)_L \neq 0$,
		  then $\alpha_i = \alpha_1$,
		   a contradiction to the simplicity of the \PF\ eigenvalue.
	      Therefore $(u_k,u_j)_L$ is non-zero iff $k=1$ and the same statement
		  holds for $(u_k,u_{j'})_L$.
	    
	      Let $u = c' u_j - c u_{j'}$ and note that
		  $$ (u_1,u)_L = c'(u_1,u_j)_L - c (u_1,u_{j'})_L = c'c - c c' = 0.$$
	      Because $(u_k,u)_L=0$ for $k\neq 1$, $c\in \NNN_L$.
	      This implies a linear dependence between $u_j$, $u_{j'}$ and $u_{n-m+1},\dots,u_n$,
		    a contradiction.
	  \end{proof}

\section{Proof of Main Theorem 1}\label{SecMain}

\begin{mainlem}
      If $\tilde{B}$ is a positive matrix defined by an extended Rauzy cycle, then the initial permutation $\pi\in \irr_n$
	    is unique.
\end{mainlem}

\begin{proof}[Proof of Main Theorem 1]
      Let $B_1,B_2,\dots$ be the matrix products defined by an infinite sequence of extended Rauzy induction steps,
	    and assume by contradiction that there exist distinct $\pi,\pi'\in \irr_n$ such that
	    infinite induction steps beginning at $\pi$ and $\pi'$ each exist and define the $B_k$'s.
      There exist $\pi_k$'s and $\pi'_k$'s, $k\in \NN_0$, so that $\pi_0 = \pi$,
	    $\pi'_0 = \pi'$ and for each $k\in \NN$,
	      $$ B_k^*L_{\pi_{k-1}}B_k = L_{\pi_k} \text{ and }B_k^*L_{\pi'_{k-1}}B_k = L_{\pi'_k}.$$
      Because the $B_k$'s are invertible and by Lemma \ref{LemLPi}, $L_{\pi_k} \neq L_{\pi'_k}$ by induction on $k$
		and so $\pi_k \neq \pi'_k$ for all $k\in \NN_0$ and are uniquely determined by
		$\pi$, $\pi'$ and the $B_k$'s.
      There exist distinct $\tilde{\pi},\tilde{\pi}'\in \irr_n$ so that $\pi_k = \tilde{\pi}$ and $\pi'_k = \tilde{\pi}'$
	    simultaneously for infinitely many $k$.
      By Lemma \ref{LemIDOC2} we may choose such $k_0,k_1$ so that $k_0<k_1$ and $\tilde{B} = B_{k_0+1} B_{k_0+2} \cdots B_{k_1}$ is positive.
      Therefore $\tilde{B}$ is a positive matrix defined by an extended Rauzy cycle at $\tilde{\pi}$
	    and also an extended Rauzy cycle at $\tilde{\pi}'$.
      By the Main Lemma $\tilde{\pi} = \tilde{\pi}'$,
	    a contradiction.
\end{proof}

\begin{proof}[Proof of Main Lemma]
      Suppose $\tilde{B}$ is a positive integer matrix that is described by two extended Rauzy cycles:
		one each at distinct $\tilde{\pi},\tilde{\pi}'\in \irr_n$.
      Then by equation \eqref{EqBLB},
	      $$ \tilde{B}^* L_{\tilde{\pi}} \tilde{B} = L_{\tilde{\pi}}
		  \text{ and }
		  \tilde{B}^* L_{\tilde{\pi}'} \tilde{B} = L_{\tilde{\pi}'}.$$
      Let $m = \dim(\NNN_{\tilde{\pi}})$, $m' = \dim(\NNN_{\tilde{\pi}'})$ and $\ell = \dim(\NNN_{\tilde{\pi}} \cap \NNN_{\tilde{\pi}'})$.
      Using Corollary \ref{CorSplit}, let $\{u_1,\dots,u_n\}$ be an eigenbasis for $B$ with respective eigenvalues $\alpha_1,\dots,\alpha_n$ such that
	      \begin{itemize}
	       \item $\alpha_1>1$ is the \PF\ eigenvalue
		      and $u_1$ is positive,
		\item $\{u_{n-m-m'+\ell+1},\dots, u_{n-m'+\ell}\}$ is a basis for $\NNN_{\tilde{\pi}}$,
		\item $\{u_{n-m'+1},\dots, u_{n-m'+\ell}\}$ is a basis for $\NNN_{\tilde{\pi}} \cap \NNN_{\tilde{\pi}'}$, and
		\item $\{u_{n-m'+1},\dots, u_n\}$ is a basis for $\NNN_{\tilde{\pi}'}$.	
	      \end{itemize}
	Because $u_1$ is not in $\NNN_{\tilde{\pi}}$ there must exist a unique
	      $j\leq n + \ell - m - m'$ so that $(u_1,u_j)_{\tilde{\pi}}\neq 0$ by Corollary \ref{CorPFUnique}.
	By Corollary \ref{CorNullspace} this is also the unique $1\leq j\leq n$
	      so that $\alpha_j = 1/\alpha_1$.
	 It follows that $(u_1,u_i)_{\tilde{\pi}'} \neq 0$ iff $i = j$ as well,
		and $j \leq n + \ell -m- m'$.

      Let $c_1 = (u_1,u_j)_{\tilde{\pi}}$, $c_2 = (u_1,u_j)_{\tilde{\pi}'}$
	      and $L = L_{\tilde{\pi}} - \frac{c_1}{c_2} L_{\tilde{\pi}'}$.
      For $i\neq j$, 
		$$(u_1,u_i)_L = (u_1,u_i)_{\tilde{\pi}} - \frac{c_1}{c_2} (u_1,u_i)_{\tilde{\pi}'} = 0 - \frac{c_1}{c_2} 0 = 0$$
	    and
		$$ (u_1,u_j)_L = (u_1,u_j)_{\tilde{\pi}} - \frac{c_1}{c_2} (u_1,u_j)_{\tilde{\pi}'} = c_1 - \frac{c_1}{c_2} c_2 = 0.$$
      This implies that $u_1\in\NNN_L$, a contradiction to Corollary \ref{CorPosNeverNull}.
\end{proof}

\section{Proof of Main Theorem 2}\label{SecMain2}

The following result is proven in \cite[Theorem 4.3]{cDolPer2013}
    under different notation.
A proof is provided in Appendix \ref{SecB}.

\begin{lem}\label{LemAdmissIsRauzy}
      Let $T:I\to I$ be an i.d.o.c. $n$-IET.
      Then $J\subsetneq I$ is admissible iff the induction realized by $J$
	    is given by consecutive steps of extended Rauzy induction.
\end{lem}

\begin{defn}
    An \emph{admissible induction sequence} for $T=\III_{\pi,\lambda}:I \to I$
	is a sequence $I',I'',\dots,I^{(k)},\dots$ so that the induction from $I^{(k-1)}$ to $I^{(k)}$
	is admissible for each $k$.
\end{defn}

By Lemma \ref{LemA3}, it follows that $|I^{(k)}|\to 0$ as $k\to\infty$ for any
      admissible induction sequence.
We say that $B_1,B_2,\dots$ are defined by an admissible induction sequence $I',I'',\dots$ if
    for each $k$ $B_k$ is the visitation matrix the induction of $T^{(k-1)}$ on $I^{(k)}$,
    where $T^{(k)}$ is the induction of $T$ on $I^{(k)}$,
    or equivalently $T^{(k)}$ is the induction of $T^{(k-1)}$ on $I^{(k)}$.

\begin{proof}[Proof of Main Theorem 2]
    A sequence of matrices $B_1,B_2,\dots$ defined by an admissible induction sequence
	must be given by products of $A_k$'s given by extended Rauzy induction by
	Lemma \ref{LemAdmissIsRauzy}.
    The result then follows from Main Theorem 1.
\end{proof}

\appendix

\section{Admissibile Inductions}\label{SecA}

      We will give a construction of the induced map given by $T:I\to I$ over
	  sub-interval $I'\subsetneq I$.
      Let $T = \III_{\pi,\lambda}$, $\pi\in \irr_n$ and $\lambda\in \RR_+^n$,
	  be an i.d.o.c. $n$-IET with sub-intervals $I_i = [\beta_{i-i},\beta_i)$, $1\leq i\leq n$,
	  where we recall $ \beta_i = \sum_{j\leq i} \lambda_j$ for $0\leq i \leq n$.
      Also, recall the return time $r(x)$ of $x\in I'$ to $I' = [a',b')$ by $T$.

      Given $x\in I'$,
      	we describe how to construct the sub-interval $I'_x\subseteq I'$ so that
      	\begin{enumerate}
      			\item $x\in I'_x$,
      			\item for each $z,z'\in I'_x$, $r(z) = r(z')$,
      			\item for each $0\leq k < r(x)$, $T$ restricted to $T^k I'_x$ is
      					a translation,
      			\item for each $0\leq k < r(x)$,
      					there exists $j=j(k)$ so that $T^k I'_x\subseteq I_j$,
      					and
      			\item $I'_x$ is the maximal sub-interval with these properties.
      	\end{enumerate}
      Compare properties 2--4 with those of intervals $I'_i$ for
      		admissibility in Definition \ref{DefAdmiss}.

      Observe that for any sub-interval $[c,d)$ such that $[c,d)\subseteq I_j$ for some $j$,
	  $T$ restricted to $[c,d)$ is a translation.
      Also, $T[c,d) \subseteq I_{j'}$ for some $j'$ iff $(c,d) \cap T^{-1}\DDD = \emptyset$ where
	    $$ \DDD = \{\beta_1,\dots,\beta_{n-1}\}.$$
      Let $a(x)$ be the minimum $y\in \III' \cap [0,x]$ so that
		\begin{itemize}
		 \item $r(z) = r(x)$ for all $z\in [y,x]$, and
		 \item for each $0\leq k <r(x)$, $T^k[y,x] \subset I_j$ for some $j = j(k)$.
		\end{itemize}
      It is an exercise to see that
	    $$x - a(x) = \min\{T^k x - z: 0\leq k <r(x),~ z\in \DDD'\cap [0,T^kx]\},$$
      where $\DDD' = \{a',b',0\}\cup\{\beta_1,\dots,\beta_{n-1}\}$.
      Therefore, either $a(x) = a'$ or $a(x) = T^{-r_(z)} z$ for some $z\in \DDD'$,
			  where $r_-(z) = \min\{k \in \NN_0: T^{-k}z \in (a',b')\}$.
      Likewise, let $b(x)$ be defined by
	    $$b(x)-x = \min\{z- T^k x: 0\leq k <r(x),~ z\in (\{|\lambda|\}\cup \DDD')\cap [T^kx, |\lambda|]\},$$
      and note that $b(x) \in (I'\cup\{b'\})\cap [x,|\lambda|]$ is the maximum value $y$ so that
		\begin{itemize}
		 \item $r(z) = r(x)$ for all $z\in [x,y)$, and
		 \item for each $0\leq k <r(x)$, $T^k[x,y) \subset I_j$ for some $j = j(k)$.
		\end{itemize}
		Either $b(x) = b'$ or $b(x) = a(x')$ for some $x'>x$.
		
      Let $I'_x = [a(x),b(x))$ for each $x\in I'$.
      Because $I'_y = I'_x$ for each $y\in I'_x$,
	      	 $$\PPP_{I'} = \{I'_x: x\in I'\}$$
      	 is a partition of $I'$.
      If $\PPP_{I'} = \{I'\}$, then $T$ is periodic and cannot be and i.d.o.c. IET.
      Therefore $\#\PPP_{I'} \geq 2$.

      \begin{lem}\label{LemA2}
	    For i.d.o.c. $n$-IET $T:I\to I$ and sub-interval $I'\subseteq I$,
		if $m = \#\PPP_{I'}$
		then $n \leq m \leq n+2$.
      \end{lem}

      \begin{proof}
      	Let $\DDD_{I'} = \{a'\}\cup\{T^{-r_-(z)}z: z\in \DDD'\}$
      			and $\gamma_j = T^{-r_-(\beta_j)}\beta_j$ for $0\leq j <n$.
      	As discussed in the previous paragraphs,
      		$a(x) \in \DDD_{I'}$ for each $x$.
      	Furthermore, $a(z) = z$ for each $a\in \DDD_{I'}$.
      	Therefore $\#\PPP_{I'} = \#\DDD_{I'}$.
      	Because $T$ is i.d.o.c., $T^\ell \beta_j = \beta_{j'}$ iff
      		$j=\pi^{-1}(1)$, $j'=0$ and $\ell=1$.
      	It follows that $\gamma_0 = \gamma_{\pi^{-1}(1)}$
      			and $\gamma_j \neq \gamma_{j'}$ for
      			distinct $j,j'>0$.
      	Therefore $\#(\DDD_{I'}\setminus\{a'\}) \geq n-1$, or $\#\PPP_{I'} \geq n$,
      				and $\#(\DDD_{I'}\setminus\{a'\}) \leq n+1$ or $\#\PPP_{I'} \leq n+2$.
      \end{proof}
      
      We order and name the sub-intervals in $\PPP_{I'}$ as $I'_1,\dots,I'_m$.
      We call this the \emph{natural decomposition} of $I'$ by $T' = T|_{I'}$,
	    and $I'$ is admissible iff $m = n$.
      The next statement proves that
	$ T|_{I''} = (T|_{I'})|_{I''}$,
      under appropriate choices of $T$, $I'$ and $I''$,
      and the natural decompositions agree with this identity.

      \begin{lem}\label{LemA1}
	    Let $T = \III_{\pi,\lambda}$ be an $n$-IET defined on $I$
		    with sub-intervals $I_1,\dots,I_n$
	      and let $\emptyset \subsetneq I'' \subsetneq I' \subsetneq I$ be
	      sub-intervals.
	    If $T' = T|_{I'}$ with natural decomposition $I_1',\dots,I_m'$ of $I'$,
		$T'' = T|_{I''}$ with natural decomposition $I_1'',\dots, I_{m'}''$ of $I''$
		and $S = T'|_{I''}$ with the natural decomposition $J_1',\dots,J_{m''}'$ of $I''$,
	     then $m'' = m'$ and $J'_i = I''_i$ for all $1\leq i \leq m'$.
      \end{lem}

      For $x\in I'$, let $q(x) = (j_0,\dots,j_{r(x)-1})$ be the ordered $r(x)$-tuple
	      given by $T^k x \in I_{j_k}$ for $0\leq k < r(x)$.
      Note that $I'_1,\dots,I'_m$ is the natural decomposition of $I'$ by $T'$
	  iff the following statements hold:
	  \begin{enumerate}
	   \item $q(x) = q(y)$ for all $x,y\in I'_i$, $1\leq i \leq m$, and
	    \item $q(x) \neq q(y)$ if $x\in I'_i$ and $y\in I'_{i'}$ for $i\neq i'$.
	  \end{enumerate}
      The definition of $q(x)$ coincides with a \emph{return word} when considering
	    an IET by its natural symbolic coding.
      See \cite{cKea1975} for an introduction to this coding.
      Because these tuples are constant on each $I'_i$, let $q_i = q(x)$ for any $x\in I'_i$.

	\begin{proof}[Proof of Lemma \ref{LemA1}]
		Let $q_i$'s be the return words of $I'$ in $I$ of $T$,
		    $q'_i$'s be the return words of $I''$ in $I$ of $T'$
		    and $q''_i$'s the return words of $I''$ in $I'$ of $S$.
		Let $q_iq_{i'}$ denote the concatenation of the tuples $q_i$ and $q_{i'}$,
		    meaning
		    $$ q_iq_{i'} = (j_0,\dots,j_{r-1},j'_0,\dots,j'_{r'-1})$$
		where $q_i = (j_0,\dots,j_{r-1})$ and $q_{i'} = (j'_0,\dots,j'_{r'-1})$.

		Let $\tilde{q}_i$, $1\leq i \leq m''$, be given by
			$$ \tilde{q}_i = q_{j_0} q_{j_1} \cdots q_{j_{r-1}}$$
		where $q''_i = (j_0,\dots,j_{r-1})$.
		These are the return words for the $J'_i$'s in $I$ by $T$.
		The lemma follows from two claims: each $\tilde{q}_i$ equals $q'_j$ for some
		    $j= j(i)$ and $\tilde{q}_i \neq \tilde{q}_{i'}$ for $i\neq i'$.
		The first claim follows as for any $x\in J'_i$ and $k \in \NN_0$,
			if $T^kx \in I'_{j'}$ then $T^{k+k'}x\in I_{j''_{k'}}$
		    for each $0\leq k'< r_{j'}$,
		    where $q_{j'} = (j''_0,\dots,j''_{r_{j'}-1})$.
		The second claim holds because $q''_i \neq q''_{i'}$ and
		    $q_j \neq q_{j'}$ for $i\neq i'$ and $j\neq j'$.
	\end{proof}

      \begin{lem}\label{LemA3}
	    Let $T = \III_{\pi,\lambda}$, $\pi\in \irr_n$ and $\lambda\in\RR_+^n$,
		be an i.d.o.c. $n$-IET.
	    If $I',I'',I''',\dots,I^{(k)},\dots$, $k\in \NN$,
		are admissible sub-intervals so that $\emptyset \subsetneq I^{(k)} \subsetneq I^{(k-1)}$
		for each $k\in \NN$, then either $J =\bigcap_{k=1}^\infty I^{(k)}$ is empty or
		consists of one point.
      \end{lem}

      \begin{proof}
	  Suppose by contradiction that $J$ is
	      contains an open interval and therefore is a sub-interval of $I$ by possibly removing
	      the right endpoint.
	  Let $J_1,\dots,J_m$ be the natural decomposition of $T|_{J}$.
	  By Lemma \ref{LemA2}, $m \leq n+2$.
	  Because $m$ is finite, $\eps= \min\{|J_i|:1\leq i \leq m\}>0$ is
	    well-defined.
	  Fix $k_0\in \NN$ so that $|I^{(k_0)}| < |J| + \eps/2$
		and let $T' = T|_{I^{(k_0)}}$.
	  By Lemma \ref{LemA1}, $J_1,\dots,J_m$ is the natural decomposition
	    of $T'$ restricted to $J$.
	  By Lemma \ref{LemMinimal}, because $J\subsetneq I^{(k_0)}$
	      there must exist $J_i$ with return time
	      in $I^{(k)}$ greater than $1$.
	  In this case $T^{(k)}$ acts on $J_i$ by translation and
		  $T^{(k)}J_i \subseteq I^{(k)} \setminus J$.
	  However, this implies that $|J_i|<\eps/2$, a contradiction.
      \end{proof}

      \section{Admissibility and Extended Rauzy Induction}\label{SecB}

      \begin{proof}[Proof of Lemma \ref{LemAdmissIsRauzy}]
	    Suppose that $J = [a,b)$ is not given by extended Rauzy induction,
		and let $S = T|_{J}$ be the induced map with
		natural decomposition $J_1,\dots,J_m$.
	    We will construct a sub-interval $I'$ of $I$ with induced $T'= T|_{I'}$
		    so that $I' = [a',b')$ is given by steps of extended Rauzy induction and if
		    $T'=\III_{\pi',\lambda'}$ then
		\begin{equation}\label{Eqaa'bb'}
		    a' \leq a < a' + \min\{\lambda'_1,\lambda'_{i_0}\}
		      \text{ and }
		    b' - \min\{\lambda'_{i_1},\lambda'_n\} < b \leq b',
		 \end{equation}
	    where $i_0,i_1$ satisfy $\pi'(i_0) = 1$ and $\pi'(i_1) = n$.
	    Because $J$ is not realized by extended Rauzy induction,
		    at most one equality may occur.
	    We then will show that $J$ is not admissible for $I'$,
		which shows our claim by Lemma \ref{LemA1}.

	    Initially, let $I' = I$ and so $J \subseteq I'$ trivially.
	    Given our definitions, suppose the inequalities for $a$ and $a'$ do
		not hold for $I$'.
	    We then perform left Rauzy induction on $T|_{I'}$ and replace $I'$ by
		this new sub-interval,
		noting that $a'$ will increase.
	    If the inequalities between $b$ and $b'$ do not hold,
		then we act by right Rauzy induction and replace $I'$ with
		this new sub-interval.
	    Note that $J\subseteq I'$ still holds after each step.
	    This process must terminate at our desired $I'$,
		otherwise we have constructed an infinite properly
		nested sequence of admissible intervals that
		each contain $J$, a contradiction to Lemma \ref{LemA3}.

	    Given $I'$ and $J$ so that inequalities \eqref{Eqaa'bb'} hold,
		$J_1,\dots,J_m$ is the natural decomposition
		from $T'$ on $J$ and the natural decomposition from $T$ on $J$ as
		well by Lemma \ref{LemA1}.
	    If $a = a'$, then $m = n+1$ and
		      $$ J_i = \RHScase{
				  I'_i, & i \leq i_1,\\
				  I'_{i_1} \setminus (T')^{-1}I_n', & i = i_1,\\
				  (T')^{-1}I_n', & i = i_1 + 1,\\
				  I'_{i-1}, & i_1 < i \leq m,}$$
	      by direct computation.
	      Similarly, $m = n+1$ if $b=b'$.
	      If both $a>a'$ and $b<b'$,
		 then analogous computations will hold unless $i_0 = n$ and $i_1 = 1$, or
		   $\pi'$ is \emph{standard}.

	      In this case, each $I'_i$ is a sub-interval in the natural decomposition of $J$ for $1<i<n$
		      with return time $1$.
	      Let $J'_1 = I'_1\cap J$ and $J'_n = I'_n \cap J$.
	      Both $J'_1 \cap (T')^{-2} J'_n$ and $J_n'\cap (T')^{-2}J'_1$ are
		    sub-intervals in the decomposition of $J$
		    with return time $2$.
	      Because $T'$ is i.d.o.c., $(T')^2 J'_1 \neq J'_n$ and $(T')^2 J'_n \neq J'_1$.
	      Therefore, there must be at least one
		    more sub-interval of $J$ in $(J'_1 \setminus (T')^{-2} J'_n)\cup (J_n'\setminus (T')^{-2}J'_1)$
	      and so $m>n$.
      \end{proof}

\bibliographystyle{abbrv}
\bibliography{../../bibfile2}

\end{document}